\documentclass{amsart}

\usepackage[noadjust]{cite}
\usepackage{color}
\usepackage{amsmath}
\usepackage{amsfonts}
\usepackage{amssymb,enumerate}
\usepackage{enumitem}
\usepackage{amsthm}
\usepackage{cite}
\usepackage{comment}
\usepackage[all]{xy}
\usepackage{hyperref}
\usepackage[dvipsnames]{xcolor}
\usepackage{lineno}

\theoremstyle{definition}
\newtheorem{lem}{Lemma}[section]
\newtheorem{cor}[lem]{Corollary}
\newtheorem{prop}[lem]{Proposition}
\newtheorem{thm}[lem]{Theorem}

\newtheorem{defn}[lem]{Definition}
\newtheorem{ex}[lem]{Example}

\newtheorem{notation}[lem]{Notation}

\newtheorem{rem}[lem]{Remark}

\renewcommand{\geq}{\geqslant}
\renewcommand{\leq}{\leqslant}

\numberwithin{equation}{section}

\begin{document}
\sloppy
\author{Saba Al-Kaseasbeh}
\address{Department of Mathematics\\
	Tafila Technical University\\
	Tafilah, Jordan}
\email[S.~Al-Kaseasbeh]{kaseabeh@ttu.edu.jo}
\author{Jim Coykendall}
\address{Department of Mathematical Sciences\\
	Clemson University\\
	Clemson, SC}
\email[J.~Coykendall]{jcoyken@clemson.edu}

\title{Adjacency-Like Conditions and Induced Ideal Graphs}
\begin{abstract}
In this paper we examine some natural ideal conditions and show how graphs can be defined that
give a visualization of these conditions. We examine the interplay between the multiplicative
ideal theory and the graph theoretic structure of the associated graph. Behavior of these graphs under standard ring extensions are studied, and in conjunction with the theory, some classical results and connections are made.
\end{abstract}

\maketitle


\section{Introduction and Background}

Recently there has been much attention paid to various aspects of commutative algebra that lend
themselves to a graph-theoretic approach. The origins of some of the aspects of this recent interest in the interplay
between graphs and commutative algebra can perhaps be traced (chronologically) to the papers
\cite{beck} and \cite{andersonlivingston}. The Anderson-Livingston paper defined the notion of
a zero divisor graph that is more commonly used today, and it has motivated an industry of work on the interplay between commutative algebra and graph theory. We recall that a zero divisor graph of a commutative ring, $R$, has
vertex set defined to be the set of nonzero zero divisors, and it is declared that there is an edge
between $a,b\in R$ if and only if $ab=0$. In \cite{andersonlivingston} the following remarkable
theorem is proved:

\begin{thm}\label{mrb}
Let $R$ be a commutative ring with identity and $Z(R)$ its zero divisor graph. Then $Z(R)$ is
connected and has diameter less than or equal to 3.
\end{thm}

The authors consider this theorem far from obvious at first blush and find it amazing that any
commutative ring with identity must have such a well-behaved structure for its zero divisor
graph.

These works have inspired myraid questions and lines of study from numerous authors. For
example, the concept of an irreducible divisor graph was introduced by the second author and J.
Maney in \cite{CMan}. Among other things, it is shown that many factorization properties, like
unique factorization or the half-factorial property, can be seen in the properties of the
irreducible divisor graph(s). Additionally, variants of the zero divisor graph have been studied by numerous authors, for example \cite{CMan} and \cite{BC2015}, and \cite{AL2016}. Of these three, \cite{BC2015} has particular relevance to the theme of this paper, as in this work, subtle variants of atomicity in integral domains are revealed by the connectedness of an associated graph. Other studies have focused on various properties of zero-divisor graphs and their variants (for example, 
the interested reader can consult \cite{ABS2016},  \cite{ACS2005}, \cite{L2018}  among others). 

Aside from zero divisor graphs, there has been much other work done in the intersection of commutative algebra and graph theory since the 1990's. In the influential paper \cite{V1990}, Cohen-Macaulay graphs were studied and (among other things) a strong connection was made between the unmixedness of an ideal of a polynomial ring associated to a complex and the Cohen-Macaulay property of the ring modulo this ideal. In  \cite{SVV1994} Rees algebras are studied with the aid of (and in parallel to) combinatorial properties of graphs naturally associated to the rings in question, and in \cite{V1995} edge ideals were developed and explored. A couple of more recent works (among many) that deal with edge ideals are \cite{HMR2022} and \cite{HKMT2021}. In short, combinatorial methods in commutative algebra, and in particular the interplay between graph theory and commutative algebra, has borne much fruit in its relatively short history and continues to be a very fertile area of research.

In this work, we adopt a slightly different approach; the motivation is to gain insight into the ideal structure of a commutative ring by assigning a graph to the set of ideals (sometimes with natural restrictions or slight variants). Three types of graphs are developed: the first types measure adjacency or ``closeness" of the ideals in $R$, the second types mimic the structure of the ideal class group and the zero divisor graphs of Anderson/Livingston, and the third types measure finite containment of ideals. 

In the third section, the adjacency graphs are given and strong connections are given with respect to the connectedness of the graphs. In particular, a complete classification is given for connectedness for all graphs and strong bounds are given for the diameters.

In the fourth section, graphs are defined that emulate the behavior of the class group of $R$, where $R$ is an integral domain. The connected/complete case is resolved and in the case that $R$ is a Dedekind domain, the connected components of the graph are studied.

In the fifth section, a graph generalizing the classical zero-divisor graph is introduced and (among other things), we capture Theorem \ref{mrb} as a corollary.

In the sixth and final section, we explore the notion of ``finite containment" graphs. As it turns out, these graphs are almost always connected (the exception being the case where $R$ is quasilocal with maximal ideal $\mathfrak{M}$ that is not finitely generated). Here the study of the diameter turns out to be a fruitful line of attack and the behavior of diameter with regard to standard extensions (polynomial and power series) as well as homomorphic images are explored. Throughout all sections, a number of examples designed to illuminate are presented.

As a final note, we introduce some notation to be used. If $I$ and $J$ are vertices of a graph, we will use the notation $I-J$ to indicate an edge between $I$ and $J$. The notation $I\leftrightarrow J$ will signify that either $I=J$ or that there is an edge between $I$ and $J$. This notation will prove convenient in the sequel.

\section{Ideal Theoretic Graphs}

In this section we define and justify some types of graphs that are determined by
ideal theoretic properties. We wish for our graphs to be simple (no loops or multiple edges), and so in the sequel when we say  that there is an edge between the vertices $u$ and $v$, we default to the condition that $u$ and $v$ are distinct. 

In general we consider, $X$, a collection of ideals of a
commutative ring, $R$, as our collection of vertices (in many cases $X$ will consist of all ideals or all nonzero ideals of
$R$). To determine an edge set we let $P(-,-)$ be a statement and
we say that there is an edge between distinct $I$ and $J$ if and only if $P(I,J)$ is true. 

We now present three definitions that explore variants of the types of ideal graphs that we investigate. The first set of definitions measure ``closeness" of the ideals of $R$ in a certain sense; they are designed to measure/highlight the density of the ideal structure of $R$.

\begin{defn}[Adjacency Graphs]
Let $R$ be a commutative ring with identity. We define the following graphs associated to $R$.
In all cases, the vertex set is the set of proper ideals of $R$, and if the ideals $I$ and $J$ have an edge between them, then $I\neq J$.
\begin{enumerate}[start=0]
\item In $GA_0(R)$, we say that $I$ and $J$ have an edge between them if and only if $I$ and $J$
are adjacent ideals.
\item In $GA_1(R)$, we say that $I$ and $J$ have an edge between them if and only if there is a
maximal ideal $\mathfrak{M}$ such that $I=J\mathfrak{M}$.
\item In $GA_2(R)$, we say that $I$ and $J$ have an edge between them if and only if
$(I:_RJ)=\mathfrak{M}$ is a maximal ideal.
\end{enumerate}
\end{defn}

The next definitions are designed to reflect the notions of ideal multiplication and class structure. The graphs $GCl_{int}(R)$ and $GCl(R)$ are designed to graphically represent variants of the ideal class group. $GZ(R)$ is an ideal-theoretic analogue of the zero-divisor graph of Anderson-Livingston. $GZ(R)$ with the $0-$ideal removed has been considered in \cite{br2011} where the terminology ``annihilating-ideal graph" is used. We will respect this terminology, but will use the notation provided for ease and uniformity of exposition.

\begin{defn}[Ideal Multiplication/Class Structure Graphs]
Let $R$ be a commutative ring with identity. We define the following graphs associated to $R$. For these definitions, the vertex set is the set of nonzero ideals of $R$ and in the case of $GZ(R)$ we also demand that the ideals are proper. As above, if the ideals $I$ and $J$ have an edge between them, then $I\neq J$.
\begin{enumerate}
\item Let $R$ be an integral domain. For $GCl_{int}(R)$, we say that $I$ and $J$ have an edge between them if and only if there is a nonzero $a\in R$ such that $I=Ja$.
\item Let $R$ be an integral domain with quotient field $K$. In $GCl(R)$, we say that $I$ and
$J$ have an edge between them if and only if there is a nonzero $k\in K$ such that $I=Jk$.
\item In $GZ(R)$ we say that $I$ and $J$ have an edge between them if and only if
$ IJ=0$.
\end{enumerate}
\end{defn}

The last set of definitions is concerned with ideals that are finitely generated (resp. principally generated) over some initial ideal. The motivation here is to investigate possible paths between ideals in a commutative ring where the steps can be done ``one generator (or a finite number of generators) at a time." Paths between two vertices (ideals) $I$ and $J$ of these graphs rely heavily on the interplay between the maximal ideals containing $I$ and the maximal ideals containing $J$. As a consequence, we will see that these graphs give insight to the structure of $\text{MaxSpec}(R)=\{\mathfrak{M}\subset R\vert \mathfrak{M}\text{ is maximal}\}$.

\begin{defn}[Finite Containment Graphs]\label{fcg}
Let $R$ be a commutative ring with identity. We define the following graphs associated to $R$;
in all cases, the vertex set is the set of proper ideals of $R$.
\begin{enumerate}
\item In $GF(R)$ we say that $I$ and $J$ have an edge between them if and only if $I\subset J$
and $J/I$ is a finitely generated ideal of $R/I$.
\item In $GP(R)$ we say that $I$ and $J$ have an edge between them if and only if $I\subset J$
and $J/I$ is a principal ideal of $R/I$.
\end{enumerate}
\end{defn}

\begin{rem}
For all of the graphs defined above (with the exceptions of $GA_2(R)$, $GCl(R)$, and $GZ(R)$) containment is used or implied in the definition. Hence one could attempt to glean more information by directing the graph, but we will take no notice of this possibility in the present work.
\end{rem}

\begin{rem}
We note that if, in the definition of $GCl_{int}(R)$, $a$ is assumed to be irreducible, then this has been investigated by J. Boynton and the second author \cite{BC2015}.
\end{rem}

\section{Adjacency: The Graphs $GA_0(R), GA_1{R},\text{and }GA_2(R)$}

In this section, we look at the graphs that mirror adjacency; as we will see the concepts of adjacency, maximal conductor, and maximal multiple are intimately related. The metrics here, in a certain sense, strive to measure how ``tightly packed" the ideals of $R$ are. Additionally, there is a certain measure of discreteness being measured when considering adjacency conditions.  

We begin by developing preliminary lemmata that will be essential later. We first note the relationship between the defining concepts of $GA_1(R)$ and $GA_2(R)$.

\begin{prop}
Let $I$ and $J$ be ideals of $R$ with $J\not\subseteq I$ and $\mathfrak{M}$ a maximal ideal. The following are equivalent.
\begin{enumerate}
\item $(I:_RJ)=\mathfrak{M}$.
\item $J\mathfrak{M}\subseteq I$.
\end{enumerate}
\end{prop}

\begin{proof}
Assuming condition (1), if $x\in\mathfrak{M}$ then $\ xJ\subseteq I$ and hence $\mathfrak{M}J\subseteq I$. 

On the other hand, if $J\mathfrak{M}\subseteq I$ then for all $x\in\mathfrak{M},\ Jx\subseteq I$ and hence $x\in(I:_RJ)$. Since we now have that $\mathfrak{M}\subseteq (I:_RJ)$ and that $\mathfrak{M}$ is maximal and $J\not\subseteq I$, then $\mathfrak{M}=(I:_RJ)$.
\end{proof}

\begin{lem}\label{colon}
Let $I, J\subseteq R$ be ideals with $J\not\subseteq I$. If $I=J\mathfrak{M}$ where $\mathfrak{M}$ is a maximal ideal then $(I:_RJ)=\mathfrak{M}$, but
not necessarily conversely. Hence $GA_1(R)$ is a subgraph of $GA_2(R)$.
\end{lem}

\begin{proof}
Note that $\mathfrak{M}J\subseteq I$, and hence $\mathfrak{M}\subseteq (I:_RJ)$. Since $J\not\subseteq I$, we must have equality.

To see that the converse does not hold in general, consider the domain $K[x,y]$, where $K$ is any field, and the ideals
$I=(x,xy,y^2)$ and $J=(x,y)$. Note that $(I:_RJ)=J$ but $J^2\subsetneq I$.
\end{proof}

\begin{lem}\label{adj}
If $I\subsetneq J$ are adjacent, then $(I:_RJ)=\mathfrak{M}$ is maximal, but not necessarily conversely. Hence $GA_0(R)$ is a subgraph of $GA_2(R)$.
\end{lem}

\begin{proof}
Let $I\subsetneq J$ be adjacent and let $\mathcal{C}=(I:_RJ)$; we will show that $\mathcal{C}$ is
maximal. Since $I$ is strictly contained in $J$, we can find an $x\in J\setminus I$;
additionally, we note that $J=(I,x)$ because of the adjacency of $I$ and $J$.

We now choose an arbitrary $z\notin\mathcal{C}$ and note that $zx\notin I$ (indeed, if $zx\in
I$ then the fact that $J=(I,x)$ would show that $z\in\mathcal{C}$ which is a contradition).
Hence, it must be the case that $J=(I,zx)$, and since $x\in J$, we obtain

\[
x=rzx+\alpha
\]

\noindent for some $r\in R$ and $\alpha\in I$. Rearranging the above, we now have

\[
x(1-rz)=\alpha\in I.
\]

Since $(1-rz)$ conducts $x$ to $I$ and $J=(I,x)$, $(1-rz)\in\mathcal{C}$. Therefore
$(\mathcal{C},z)=R$ for all $z\notin\mathcal{C}$ and so $\mathcal{C}$ is maximal.

To see that the converse does not hold, we revisit the example of the previous result. Letting $R=K[x,y]$, $A=(x^2,xy,y^2), I=(x,xy,y^2), J=(x,y)$. Note that $(A:_RJ)=J$ but $A$ and $J$ are not adjacent as $A\subsetneq I\subsetneq J$.
\end{proof}

\begin{lem}
Suppose that $I\subsetneq J$ are adjacent and $\mathcal{C}:=(I:_RJ)$. If $\mathcal{C}\bigcap J=I$
then $J\setminus I$ is a multiplicatively closed set.
\end{lem}

\begin{proof}
Let $x,y\in J\setminus I$. Certainly $xy\in J$. By way of contradiction, we assume that $xy\in
I$. Since $I$ and $J$ are adjacent and $x\notin I$, $(x,I)=J$.

Now let $j\in J$ be arbitrary. By the previous remark, we can find $r\in R$ and $i\in I$ such
that $j=rx+i$. Multiplying this by $y$ we obtain that $yj=rxy+iy\in I$. Hence
$y\in\mathcal{C}\bigcap J=I$ which is the desired contradiction.
\end{proof}

\begin{lem}\label{intersection}
Let $I\subsetneq J$ be adjacent and $A\subseteq R$ an ideal. Then the ideals $I\bigcap A$ and
$J\bigcap A$ are either equal or adjacent.
\end{lem}

\begin{proof}
We will assume that $I\bigcap A$ and $J\bigcap A$ are distinct and suppose that $x\in (J\bigcap
A)\setminus (I\bigcap A)$. Since $x\notin I$, it must be the case that $J=(I,x)$.

Let $j\in J\bigcap A$ be arbitrary. Since $J=(I,x)$, we have that 

\[
j=rx+\alpha
\]

\noindent for some $r\in R$ and $\alpha\in I$. We observe further that $rx\in J\bigcap A$, and
hence, $\alpha\in I\bigcap A$. We conclude that

\[
J\bigcap A=(I\bigcap A, x)
\]

\noindent for any $x\in (J\bigcap A)\setminus (I\bigcap A)$, and so $I\bigcap A$ and $J\bigcap
A$ must be adjacent.
\end{proof}

\begin{lem}\label{adjunction}
If $I\subsetneq J$ are adjacent ideals and $x\in R$, then the ideals $(I,x)$ and $(J,x)$ are
either equal or adjacent.
\end{lem}

\begin{proof}
Suppose that $A$ is an ideal with $(I,x)\subsetneq A\subseteq (J,x)$ and let $a\in A\setminus
(I,x)$. We write $a=j+rx$ with $j\in J\setminus I$. Since $I$ and $J$ are adjacent and $j\notin
I$, $(I,j)=J$. As $A$ contains $x$, $A$ contains all of $J$; we conclude that $A=(J,x)$ and the proof is
complete.
\end{proof}

\begin{prop}\label{local}
Let $I\subsetneq J$ be adjacent ideals and $S\subset R$ be a multiplicatively closed set (not
containing $0$). Then the ideals $I_S\subseteq J_S$ are either adjacent or equal.
\end{prop}

\begin{proof}
Since $I\subsetneq J$ are adjacent, there is an $x\in J\setminus I$ such that $J=(I,x)$ (in
fact, {\it any} $x\in J\setminus I$ will do). Suppose that $I_S\subsetneq J_S$ and let
$\frac{b}{s}\in J_S\setminus I_S$. We can assume that $b\in J\setminus I$ and since
$I\subsetneq J$ are adjacent, we can find an $r\in R$ and $\alpha\in I$ such that
$b=rx+\alpha$. In $R_S$ we have the equation

\[
\frac{b}{s}=(\frac{r}{s})(\frac{x}{1})+\frac{\alpha}{s}
\]

\noindent which shows that $J_S$ is generated over $I_S$ by any element of the form $\frac{x}{1}$ with $x\in J\setminus I$. Hence $I_S$ and $J_S$ are either equal or adjacent.
\end{proof}

The next result is a consequence of the Lattice Isomorphism Theorem and the one after is straightforward; we omit the proofs.

\begin{prop}\label{mod}
Let $A,B\subseteq R$ be ideals containing the ideal $I$. Then $A$ and $B$ are adjacent in $R$
if and only if $A/I$ and $B/I$ are adjacent in $R/I$.
\end{prop}

\begin{lem}\label{fg}
Let $R$ be a ring and $I\subseteq J$ be ideals. If $I$ is finitely generated and $J/I\subseteq
R/I$ is finitely generated, then $J$ is finitely generated.
\end{lem}

\begin{prop}\label{na}
Let $R$ be a $1-$dimensional domain. $R$ is Noetherian if and only if $R/I$ is Artinian for
each ideal $0\neq I\subseteq R$.
\end{prop}

\begin{proof}
Suppose first that $R$ is Noetherian. If $I\subseteq R$ is a nonzero ideal, then $R/I$ is
Noetherian of dimension $0$ and hence is Artinian.

Now suppose that $R/I$ is Artinian for each nonzero ideal $I\subset R$. It suffices to show
that every ideal of $R$ is finitely generated. Let $J\subset R$ be an arbitrary nonzero ideal.
Let $0\neq x\in J$ and note that by hypothesis, $R/(x)$ is Artinian and so $J/(x)$ is finitely generated. From Lemma \ref{fg} it
follows that $J$ is finitely generated and hence $R$ is Noetherian.
\end{proof}

We now apply these preliminary results to the ideal graphs in question. We begin with the following result that records the stability of $G_i(R)$ for $0\leq i\leq 2$ with respect to localization and homomorphic images.

\begin{thm}
Let $I\subset R$ be an ideal and $S\subset R$ be a multiplicatively closed set. If $GA_i(R)$ is
connected or complete for $0\leq i\leq 2$, then so are $GA_i(R_S)$ and $GA_i(R/I)$.
\end{thm}

\begin{proof}
The case for $GA_0(-)$ follows immediately from Proposition \ref{local} and Proposition
\ref{mod}.

For $GA_1(-)$ we suppose that $A, B\subset R$ are distinct and have an edge between them. Then $A=B\mathfrak{M}$, with $\mathfrak{M}$ maximal. Note that if $\mathfrak{M}\bigcap S\neq\emptyset$ then $\mathfrak{M}_S=R_S$ and if not then $\mathfrak{M}_S\subset R_S$ is maximal. As $A_S=B_S\mathfrak{M}_S$, we have that $A_S$ and $B_S$ are either equal or have an edge between them. Since all ideals of $R_S$ are of this form, if $GA_1(R)$ is either connected or complete, then so if $GA_1(R_S)$. Additionally, if $I\subseteq A,B$ then $A/I=(B/I)(\mathfrak{M}/I)$ and from this the result follows for $GA_1(R/I)$.

In a similar fashion, if we consider $GA_2(-)$ with the ideals/notation in the previous paragraph, we assume that $(A:_RB)=\mathfrak{M}$. In this case, it is easy to verify that $(A/I:_RB/I)=\mathfrak{M}/I$ and $\mathfrak{M}_S\subseteq(A_S:_RB_S)$. Hence, again, the result follows for $GA_2(-)$.
\end{proof}

We remark that the previous result will be an immediate consequence of later structure theorems concerning $GA_i(R)$ (see Theorem \ref{art} and Theorem \ref{conn}).

The following result is of some independent interest and will be crucial in later work where we show the connection between $GA_0(R)$ and the Artinian condition. The result shows that if $I\subseteq J$ are ideals and if there is a finite chain of adjacent ideals connecting $I$ and $J$, this chain can be refined to a finite increasing chain.

\newpage

\begin{prop}\label{chainrefine}
Let $I\subseteq J$ be ideals and $\{I_n\}_{n=0}^N$ ideals with $I=I_0$, $J=I_N$, and $I_k$ and
$I_{k+1}$ adjacent for each $0\leq k\leq N-1$. Then there exists an increasing chain of ideals
\[
I=J_0\subseteq J_1\subseteq\cdots\subseteq J_M=J
\]

\noindent with each successive pair of ideals adjacent and $M\leq N$.
\end{prop}

\begin{proof}
Using the notation above, we say that the ideal $I_k$ is a {\it hinge ideal} if $I_k$ either
properly contains both $I_{k-1}$ and $I_{k+1}$ or is properly contained in them both. We
proceed by induction on $m$, the number of hinge ideals between $I$ and $J$.

At the outset, we simplify matters by assuming that all of the ideals $\{I_s\}$ are contained in $J$ by intersecting the collection with $J$
and applying Lemma \ref{intersection}; we also note that if this chain is increasing, then the conclusion holds. We now describe a reduction process that will greatly simplify the inductive argument. 

We first suppose that the first hinge ideal is contained in $I$; that is, we have the decreasing chain of ideals

\[
I:=I_0\supseteq I_1\supseteq\cdots\supseteq I_t.
\]

Because
of the adjacency of successive elements of the chain, we can find $x_s\in I_s\setminus
I_{s+1},\ 0\leq s\leq t-1$ such that

\[
I_s=(I_{s+1},x_s).
\]

We now make a preliminary refinement of the collection $\{I_s\}_{s=0}^N$ by letting $I_s^{\prime}=(I_s,x_0,x_1,\cdots ,x_{t-1})$ for $0\leq s\leq N$. Applying Lemma \ref{adjunction} $t$ times to the entire collection of ideals shows
that the new collection $\{I_s^{\prime}\}$ begins as

\[
I=I_0^{\prime}=I_1^{\prime}=\cdots=I_t^{\prime}\subseteq I_{t+1}^{\prime}
\]

\noindent and the ideals $\{I_s^{\prime}\}_{s=t}^{N}$ are successively adjacent (or
equal). Additionally if $I_j$ is an ideal in the collection that is not initially a hinge ideal, then $I_j^{\prime}$ can neither be properly contained in, nor properly contain, both the ideals $I_{j-1}^{\prime}$ and $I_{j+1}^{\prime}$. So when we identify ideals that are equal, our refined collection has less than or equal to the initial number of ideals and at least one less hinge ideal.

With this in hand, we proceed to argue inductively on the number, $m$, of hinge ideals appearing between $I$ and $J$. If $m=0$ then the chain is increasing and the conclusion is immediate, and this, coupled with the above argument, gives the case $m=1$. We now suppose that the conclusion holds for $m\geq 0$ and consider the case of $m+1$ hinge ideals. 

In the first case, we assume that the first hinge ideal $I_t$ is contained in $I$. The previous argument shows that we can refine so that $I=I_t^{\prime}$, and as above, the number of ideals in the new collection (after equal ideals are identified) is nonincreasing and the number of hinge ideals in the new collection is less than or equal to $m$. By induction, we are done in this case.

In the second case, we suppose that the first hinge ideal $I_t$ contains $I$. In this case, we consider the collection $\{I_s\}_{s=t}^N$. Since this collection has $m$ hinge ideals with the first being contained in $I_t$, we apply the first case and extract the increasing chain $\{I_s^{\prime}\}_{s=0}^k$ with $I_0^{\prime}=I_t$ with $k\leq N-t$. Combining this increasing chain with $\{I_r\}_{r=0}^t$ (and subtracting $1$ for the repetition of $I_t$), we obtain our increasing chain with length $t+k+1\leq N-t+1+t=N+1$.
\end{proof}

We now characterize some types of rings via their induced adjacency graphs; we begin with $GA_0(-)$.

\begin{thm}\label{art}
Let $R$ be a commutative ring with identity and $GA_0(R)$ its adjacency graph. The following are
equivalent.
\begin{enumerate}
\item $R$ is Artinian.
\item $GA_0(R)$ is connected.
\end{enumerate}
\end{thm}

\begin{proof}
If $R$ is Artinian, then every finitely generated $R-$module (in particular, an ideal of $R$) has a composition series. Since every ideal of $R$ is connected to $(0)$, we have that $GA_0(R)$ is connected.

On the other hand, if $GA_0(R)$ is connected, then there is a (finite) path between $(0)$ and any ideal of $R$. Proposition \ref{chainrefine} allows this to be refined to a composition series, and hence $R$ is Artinian.
\end{proof}

We now produce the following corollary with a slight variant. We define the graph $GA^*_0(R)$ to be the subgraph of $GA_0(R)$ with the $0-$ideal vertex removed.

\begin{cor}
Let $R$ be an integral domain that is not a field. $GA^*_0(R)$ is connected if and only if $R$ is Noetherian and
$\text{dim}(R)=1$.
\end{cor}

\begin{proof}
Suppose that $R$ is $1-$dimensional and Noetherian and let $I, J\subseteq R$ be nonzero ideals.
To show that $GA^*_0(R)$ is connected, it suffices to show that $I$ and $I\bigcap J$ can be
connected in a finite sequence of steps. To this end, we note that since $I\bigcap J$ is
nonzero, the ring $R/(I\bigcap J)$ is $0-$dimensional and Noetherian, and hence Artinian. By Theorem \ref{art} there is a finite sequence
of adjacent ideals (of the displayed form)

\[
(I\bigcap J)/(I\bigcap J)\subset I_1/(I\bigcap J)\subset I_2/(I\bigcap J)\subset\cdots\subset
I/(I\bigcap J)
\]

\noindent connecting $I/(I\bigcap J)$ to the zero ideal in $R/(I\bigcap J)$. By Proposition
\ref{mod} this corresponds to a chain of adjacent ideals

\[
(I\bigcap J)\subset I_1\subset I_2\subset\cdots\subset I
\]

\noindent in $R$ and hence $GA^*_0(R)$ is connected.

Now we suppose that $GA^*_0(R)$ is connected. Let $I\subset R$ be an arbitrary nonzero ideal.
Proposition \ref{mod} assures us that adjacency is preserved modulo $I$ and so we obtain that
$GA_0(R/I)$ is connected. Hence Theorem \ref{art} gives us that $R/I$ is Artinian (for any
nonzero ideal $I$). From Proposition \ref{na} we obtain that $R$ is Noetherian.

To see that $R$ is $1-$dimensional, we suppose that there is a chain of primes
$(0)\subsetneq\mathfrak{P}\subsetneq\mathfrak{M}$. If we choose the ideal $I=\mathfrak{P}$, we would have that $\text{dim}(R/\mathfrak{P})\geq 1$ and hence $R/\mathfrak{P}$ is not
Artinan. We conclude that $GA_0(R/\mathfrak{P})$ is not connected and again Proposition \ref{mod} gives us that $GA^*_0(R)$ cannot be connected and so $\text{dim}(R)\leq 1$. Because of the fact that $R$ is a domain that is not a field, then the dimension of $R$ is precisely $1$.
\end{proof}

\begin{prop}\label{g2conn}
If $GA_0(R)$ or $GA_1(R)$ is connected, then so is $GA_2(R)$.
\end{prop}

\begin{proof}
Since the vertex sets are the same, this is immediate from Lemma \ref{colon} and Lemma \ref{adj}.
\end{proof}

We now present a strong characterization of connectedness for the graphs $GA_i(R)$, $i=1,2$. Of course this condition is weaker than the Artinian condition.

\begin{thm}\label{conn}
Let $R$ be a commutative ring with identity. The following conditions are equivalent.
\begin{enumerate}
\item $GA_1(R)$ is connected.
\item $GA_2(R)$ is connected.
\item There is a collection of not necessarily distinct maximal ideals
$\{\mathfrak{M}_1,\mathfrak{M}_2,\cdots,\mathfrak{M}_n\}$ such that
$\mathfrak{M}_1\mathfrak{M}_2\cdots\mathfrak{M}_n=0$.
\end{enumerate}
\end{thm}

\begin{proof}
By Proposition \ref{g2conn}, we have the implication $(1)\Longrightarrow (2)$. For the
implication $(2)\Longrightarrow (3)$, we suppose that $GA_2(R)$ is connected and that
$\mathfrak{M}$ is a maximal ideal of $R$. By assumption, there is a finite path from
$\mathfrak{M}$ to the ideal $(0)$:

\[
\mathfrak{M}=I_0\supsetneq I_1\bowtie I_2\bowtie\cdots\bowtie
I_m\supsetneq I_{m+1}=0
\] 
\noindent where each $\bowtie$ denotes either $\supsetneq$ or $\subsetneq$. In the
proof of this implication, we will use the notion of hinge ideals introduced in the proof of
Proposition \ref{chainrefine}. Note that there must be an even number of hinge ideals in the
path described above, which we will denote $H_1, H_2,\cdots, H_{2t}$. So an abbreviated version
of the path described above can be expressed in the form

\[
\mathfrak{M}\supsetneq H_1\subsetneq H_2\supsetneq\cdots\subsetneq H_{2t}\supsetneq 0
\]

\noindent where we will have the convention that $H_j=I_{s_j}$ for all $1\leq j\leq 2t$. We
also declare that $s_0=0$ and $s_{2t+1}=m+1$.

Since this is a path in the graph $GA_2(R)$, then successive ideals must have maximal conductor.
We will say that $M_k=(I_{k+1}:I_k)$ if $s_{2a}\leq k\leq s_{2a+1}$ for $0\leq a\leq t$.

We first note that

\[
H_1=I_{s_1}\supseteq \mathfrak{M}M_0M_1\cdots M_{s_1-1}
\]

\noindent and since $H_1\subseteq H_2$, we have that

\[ 
H_2\supseteq\mathfrak{M}M_0M_1\cdots M_{s_1-1}.
\]

In a similar fashion, we note that 

\[
H_3=I_{s_3}\supseteq H_2M_{s_2}M_{s_2+1}\cdots M_{s_3-1}\supseteq\mathfrak{M}M_0M_1\cdots
M_{s_1-1}M_{s_2}M_{s_2+1}\cdots M_{s_3-1}.
\]

Inductively we obtain

\[
H_{2i+1}\supseteq\mathfrak{M}M_0M_1\cdots M_{s_1-1}M_{s_2}M_{s_2+1}\cdots M_{s_3-1}\cdots
M_{s_{2i}}M_{s_{2i}+1}\cdots M_{s_{2i+1}-1}.
\]

\noindent In particular we obtain 

\[
0=\mathfrak{M}M_0M_1\cdots M_{s_1-1}M_{s_2}M_{s_2+1}\cdots M_{s_3-1}\cdots
M_{s_{2t}}M_{s_{2t}+1}\cdots M_{m}
\]

\noindent and hence there is a collection of maximal ideals with product equal to $(0)$.

For the implication $(3)\Longrightarrow (1)$, we will assume that there is a collection of
maximal ideals $\mathfrak{M}_i,\ 1\leq i\leq n$ such that
$\mathfrak{M}_1\mathfrak{M}_2\cdots\mathfrak{M}_n=0$. To show that $GA_1(R)$ is connected, it
suffices to show that if $I\subset R$ is an arbitrary ideal, then there is a finite path to the
zero ideal. But note that

\[
I\supseteq I\mathfrak{M}_1\supseteq I\mathfrak{M}_1\mathfrak{M}_2\supseteq\cdots\supseteq
I\mathfrak{M}_1\mathfrak{M}_2\cdots\mathfrak{M}_n=0
\]

\noindent is such a path of length no more than $n$.
\end{proof}

\begin{cor}
If $GA_i(R),\ 1\leq i\leq 2$ is connected then $\text{diam}(GA_i(R))\leq 2n$ where $n$ is the
smallest positive integer for which there is a collection of maximal ideals $\mathfrak{M}_i,
1\leq i\leq n$ for which $\mathfrak{M}_1\mathfrak{M}_2\cdots\mathfrak{M}_n=0$.
\end{cor}

\begin{proof}
The fact that $\text{diam}(GA_1(R))\leq 2n$ is immediate from the proof of the implication
$(3)\Longrightarrow (1)$ in Theorem \ref{conn}. The fact that $\text{diam}(GA_2(R))\leq 2n$
follows from the fact that $GA_1(R)$ is a subgraph of $GA_2(R)$ (see Proposition \ref{g2conn}).
\end{proof}

\begin{cor}\label{q0}
If $GA_i(R)$ is connected for some $0\leq i\leq 2$ then $R$ is semiquasilocal and $0-$dimensional.
\end{cor}

\begin{proof}
By Theorem \ref{art}, $GA_0(R)$ is connected if and only if $R$ is Artinian, and hence $R$ is
$0-$dimensional, and, in this case, semilocal. If $GA_1(R)$ or $GA_2(R)$ is connected then
Theorem \ref{conn} gives that $\mathfrak{M}_1\mathfrak{M}_2\cdots\mathfrak{M}_n=0$ for (not
necessarily distinct) ideals $\{\mathfrak{M}_1,\mathfrak{M}_2,\cdots,\mathfrak{M}_n\}$. If
$\mathfrak{M}$ is an arbitrary maximal ideal, then
$\mathfrak{M}\supseteq\mathfrak{M}_1\mathfrak{M}_2\cdots\mathfrak{M}_n$ and hence
$\mathfrak{M}=\mathfrak{M}_k$ for some $1\leq k\leq n$ which shows that the list of ideals
$\{\mathfrak{M}_1,\mathfrak{M}_2,\cdots,\mathfrak{M}_n\}$ contains $\text{MaxSpec}(R)$. Hence
$R$ is semiquasilocal.

To see that $R$ is $0-$dimensional, we appeal once again to the fact that 

\[
\mathfrak{M}_1\mathfrak{M}_2\cdots\mathfrak{M}_n=0.
\]

Recalling that this collection of maximal ideals contains $\text{MaxSpec}(R)$, we suppose that
we can find a prime ideal $\mathfrak{P}$ such that $\mathfrak{M}_k\supsetneq\mathfrak{P}$.
Since $\mathfrak{P}\supseteq\mathfrak{M}_1\mathfrak{M}_2\cdots\mathfrak{M}_n=0$, we must have
that $\mathfrak{P}\supseteq\mathfrak{M}_i$ for some $1\leq i\leq n$. Hence
$\mathfrak{M}_i\subsetneq\mathfrak{M}_k$ which is our desired contradiction.
\end{proof}

As a companion to Proposition \ref{g2conn}, we present the following corollary.

\begin{cor}
If $GA_0(R)$ is connected, then both $GA_1(R)$ and $GA_2(R)$ are connected. 
\end{cor}

\begin{proof}
Combine the results from Proposition \ref{g2conn} and Theorem \ref{conn}.
\end{proof}

As a final observation, we consider the following.

\begin{cor}
Let $R$ be a commutative ring with identity. We consider the following conditions:

\begin{enumerate}
\item $GA_0(R)$ is connected.
\item $GA_1(R)$ is connected.
\item $GA_2(R)$ is connected.
\end{enumerate}

If $R$ is Noetherian, the above three conditions are equivalent.
\end{cor}

\begin{proof}
We know that (2) and (3) are equivalent under any conditions and that condition (1) implies (2) and (3). Suppose that $R$ is Noetherian and any one of the conditions hold. By Corollary \ref{q0}, $R$ is $0-$dimensional. So if we add ``Noetherian" as a hypothesis, then $R$ is Artinian and all three conditions hold.
\end{proof}

We conclude this section with an example of a ring $R$ for which $GA_0(R)$ is not connected, but $GA_i(R)$ is connected for $i=1,2$.

\begin{ex}
Consider the ring

\[
R:=\mathbb{F}[x_1,x_2,x_3,\cdots]/(\{x_ix_j\}_{i,j\geq 1}).
\]

This ring is not Noetherian (and hence not Artinian) so $GA_0(R)$ is not connected. On the other hand, $GA_1(R)$ and $GA_2(R)$ are connected. To see this, note that the unique maximal ideal of $\mathfrak{M}\subseteq R$ (generated by the image of the elements $x_i$) has the property that $\mathfrak{M}^2=0$. Hence every nonzero ideal of $R$ is connected by an edge to $(0)$ in both $GA_1(R)$ and $GA_2(R)$.
\end{ex}

\section{Class Structure: $GCl_{int}(R)$ and $GCl(R)$}

For the graphs $GCl(R)$ and $GCl_{int}(R)$ we assume that $R$ is an integral domain with quotient field $K$ unless specified otherwise.

It should be noted that in the case that $R$ is an integral domain, $GCl_{int}(R)$ is a variant on
the so-called divisor graph of an integral domain studied in \cite{BC2015}, where the ideals
$I$ and $J$ are assumed to be principal and possess an edge between them if $I=Ja$ where $a\in
R$ is irreducible.

For this section, it will also be useful to keep in mind that $GCl_{int}(R)$ is a subgraph of $GCl(R)$.

\begin{thm}\label{g4}
The following conditions are equivalent.
\begin{enumerate}
\item $GCl(R)$ is connected.
\item $GCl(R)$ is complete.
\item $R$ is a PID
\end{enumerate}
\end{thm}

\begin{proof}
For this proof, we discard the case where $R$ is a field as all of the conditions are satisfied
vacuously.
Since any complete graph is connected, $(2)\Longrightarrow (1)$ is immediate.
For the implication $(1)\Longrightarrow (3)$, we let $I\subset R$ be an arbitrary ideal and
$R$ the unit ideal. Since $GCl(R)$ is connected,
there is a sequence of ideals connecting $R$ and $I$:

\[
R:=J_0\relbar J_1\relbar J_2\relbar\cdots\relbar J_{n-1}\relbar I:=J_n.
\]

Since the edges above are in the graph $GCl(R)$, we must have, for all $1\leq i\leq n$, nonzero $k_i\in
K$ such that $J_{i-1}=k_iJ_{i}$. Note that $R=k_1k_2\cdots k_nI$ and hence, $I$ is principal.

Finally for the implication $(3)\Longrightarrow (2)$, we let $I=aR$ and $J=bR$ be two arbitrary
ideals of $R$ (with $a,b\neq 0$). Note that $I=\frac{a}{b}J$ and hence $GCl(R)$ is complete.
\end{proof}

\begin{thm}\label{g3com}
$GCl_{int}(R)$ is complete if and only if $R$ is a Noetherian valuation domain. 
\end{thm}

\begin{proof}
In this proof, we will ignore the simple case where $R$ is a field.
For the forward implication, we will assume that $GCl_{int}(R)$ is complete. As $GCl_{int}(R)$ is a
subgraph of $GCl(R)$, $GCl(R)$ must also be complete. By Theorem \ref{g4}, $R$ must be a
PID. It now suffices to show that $R$ is local. To this end, suppose that $\frak{M}_1\neq \mathfrak{M}_2$ are maximal ideals. Without loss of generality, there is a nonzero 
$a\in R$ such that $\frak{M}_1=a\frak{M}_2$. If $a$ is a unit then $\mathfrak{M}_1=\mathfrak{M}_2$ and if $a$ is a nonunit then $\mathfrak{M}_1=a\mathfrak{M}_2\subsetneq\mathfrak{M}_2$; either gives a contradiction. Hence $R$ is a local PID and hence a Noetherian valuation domain.

On the other hand, if $R$ is a Noetherian valuation domain then any two nonzero proper ideals
are of the form $(\pi^n)$ and $(\pi^m)$ where $\pi$ is a generator of the maximal ideal and
$n,m\geq 1$. If we say (without loss of generality) that $n\leq m$ then
$\pi^{m-n}(\pi^n)=(\pi^m)$ and hence $GCl_{int}(R)$ is complete.
\end{proof}

\begin{thm}
$GCl_{int}(R)$ is connected if and only if $R$ is a PID. In this case, $\text{diam}(GCl_{int}(R))\leq
2$, and $\text{diam}(GCl_{int}(R))=1$ if and only if $R$ is local.
\end{thm}

\begin{proof}
As $GCl_{int}(R)$ is a subgraph of $GCl(R)$, the fact that $GCl_{int}(R)$ is connected implies that
$GCl(R)$ is connected. Hence, by Theorem \ref{g4}, $R$ must be a PID.

On the other hand, if $R$ is a PID and $I=aR$ and $J=bR$ ($a,b\neq 0$) are arbitrary ideals,
then we can connect $I$ and $J$ as follows:

\[
I=aR\leftrightarrow abR\leftrightarrow bR=J.
\]

The above demonstrates the veracity of the remark that $\text{diam}(GCl_{int}(R))\leq 2$. The fact
that $\text{diam}(GCl_{int}(R))=1$ precisely when $R$ is local follows from Theorem \ref{g3com}.
\end{proof}

\begin{thm}\label{ded}
If $R$ is a Dedekind domain with quotient field $K$, then the connected components of the graphs $GCl(R)$ and $GCl_{int}(R)$ are in one to one correspondence with the elements of the class group $\text{Cl}(R)$. Each connected component of $GCl(R)$ is complete and each connected component of $GCl_{int}(R)$ has diameter no more than 2 and the connected components of $GCl_{int}(R)$ are complete if and only if $R$ is a local PID.
\end{thm}

\begin{proof}
If $R$ is a Dedekind domain with quotient field $K$, then two ideals, $I$ and $J$, are in the same class of $\text{Cl}(R)$ if and only if $I=Jk$ for some nonzero $k\in K$. Hence each connected component of $GCl(R)$ is complete and these components are in one to one correspondence with the elements of $\text{Cl}(R)$.

For the $GCl_{int}(R)$ case, we first note that if there is a path from $I$ to $J$ then there must be some nonzero $k\in K$ such that $I=Jk$, so it remains to show that if $I$ and $J$ are in the same ideal class, then there is a path connecting them. To this end, we note that if $I=Jk$ for some nonzero $k\in K$, we can write $I=\frac{a}{b}J$ and in a similar fashion as before, we can connect $I$ and $J$ via

\[
I\leftrightarrow  bI=aJ\leftrightarrow J,
\]

\noindent and hence there is a path of length no more than 2 connecting $I$ and $J$. 

For the final statement, we first suppose that $R$ is a local PID (and hence a Noetherian valuation domain) with uniformizer $\pi$. If $I$ and $J$ are in the same class of $\text{Cl}(R)$ then $I=kJ$ for some nonzero $k\in K$. Since $K$ is the quotient field of a Noetherian valuation domain, we can write $k=\pi^n$ with $n\in\mathbb{Z}$ up to a unit in $R$. If $n\geq 0$ then $I=J$ or there is an edge between them. If $n<0$ then $J=\pi^{-n}I$ and again there is an edge between $I$ and $J$.

On the other hand, suppose that some connected component of $GCl_{int}(R)$ is complete and that there are two maximal ideals $\mathfrak{M}_1, \mathfrak{M}_2\subset R$. Select $m_1\in\mathfrak{M}_1\setminus\mathfrak{M}_2$ and $m_2\in\mathfrak{M}_2\setminus\mathfrak{M}_1$. If $I$ is in a complete connected component of  $GCl_{int}(R)$, then there is a path from $m_1I$ to $m_2I$ and hence there must be an edge between them. Hence we have $m_1\vert m_2$ or $m_2\vert m_1$ and in either case, we obtain a contradiction. We conclude that $R$ is local and hence a PID (Noetherian valuation domain).
\end{proof}

In the spirit of these results, we further restrict the set of vertices to make a more general observation. We define $GCl_{inv}(R)$ to be the subgraph of $GCl(R)$ with the vertex set restricted to the set of invertible ideals. Recall that if $R$ is a domain with quotient field $K$, then $R$ is seminormal if for all $\omega\in K$, $\omega^2,\omega^3\in R$ implies that $\omega\in R$ (see \cite{swan} for a good reference on this topic). The next result shows that if $R$ is a seminormal domain then the number of connected components in $GCl_{inv}(R)$ is stable for polynomial extensions.

\begin{thm}
If $R$ is a seminormal domain, then the number of connected components of $GCl_{inv}(R)$ is equal to the number of connected components of \\ $GCl_{inv}(R[x_1,x_2,\cdots,x_n])$ for all $n\geq 1$.
\end{thm}

\begin{proof}
This result follows from the observations that in the integral domain case $\text{Pic}(R)$ is isomorphic to $\text{Cl}(R)$ and $\text{Pic}(R)\cong\text{Pic}(R[x])$ (\cite{GH1980}) and the fact that a polynomial extension of a seminormal domain is seminormal (\cite{swan}). With these results in hand, the fact that we are restricting the vertex set to the set of invertible ideals makes the rest of the proof almost identical to the proof of Theorem \ref{ded}.
\end{proof}

\section{$GZ(R)$ and the classical zero-divisor graph}

In this section, our attention will be devoted to the graph
$GZ(R)$ and some of its variants. The reason for excluding the zero ideal is precisely the same as the reason that this exception was made in \cite{andersonlivingston}: namely because the use
of the zero ideal gives extra structure to this graph with no useful new information. It is
easy to see that for any commutative ring with identity (even an integral domain) that the
graph $GZ(R)$ is connected with diameter no more than $2$ if we allow use of the zero ideal.
Indeed, if $I$ and $J$ are arbitrary ideals, then $I\leftrightarrow (0)\leftrightarrow J$ is a path connecting
them. So if $R$ is an (infinite) integral domain, the graph $GZ(R)$ would be an infinite star graph with
the zero ideal at the center. These extra connections muddy the waters and give no useful
insights for our current purposes.

The first variant that we will consider is the graph $GZ^*(R)$. For this graph, the set of vertices is the collection of nonzero ideals $I\subset R$ such that there is a nonzero ideal $J\subset R$ such that $IJ=0$. As before, we say that $I$ and $J$ have an edge between them if $IJ=0$. Of course, if $R$ is a domain, this produces the empty graph, so in this situation, $R$ must have nontrivial zero divisors for this graph to be of any interest.

We begin with a well-known lemma concerning annihilators that we record here for completeness with proof omitted.

\begin{lem}\label{ann}
Let $I\subseteq J$ be ideals in $R$.
\begin{enumerate}
\item $\text{Ann}(J)\subseteq\text{Ann}(I)$.
\item $I\subseteq\text{Ann}(\text{Ann}(I))$.
\item $\text{Ann}(I)=\text{Ann}(\text{Ann}(\text{Ann}(I)))$.
\end{enumerate}
\end{lem}

The next theorem reminiscent of Theorem \ref{mrb} from \cite{andersonlivingston} couched in an ideal-theoretic setting. This theorem can also be found in \cite{br2011} and may be considered a consequence of \cite[Theorem 11]{dd2005}.  This will be used to leverage some further insights and to provide an alternate proof of the first statement in \cite[Theorem 2.3]{andersonlivingston}.

\begin{thm}\label{gen}
Let $R$ be a commutative ring with 1, then $GZ^*(R)$ is connected and has diameter no more than 3.
\end{thm}

\begin{proof}
Let $I$ and $J$ be distinct nonzero ideals of $R$ such that there are ideals $0\neq I_1, J_1\subset R$ with $II_1=0=JJ_1$. As before, we use the notation ``$X \leftrightarrow Y$" to mean that $X=Y$ or that there is an edge between $X$ and $Y$.

We first note that if $J\text{Ann}(I)\neq 0$ then we have the path

\[
I\leftrightarrow J\text{Ann}(I)\leftrightarrow\text{Ann}(J)\leftrightarrow J
\]

\noindent and similarly, if $I\text{Ann}(J)\neq 0$ we have the path

\[
I\leftrightarrow\text{Ann}(I)\leftrightarrow I\text{Ann}(J)\leftrightarrow J
\]

So the only case to consider is the case where $J\text{Ann}(I)=I\text{Ann}(J)=0$. Note that this implies that $\text{Ann}(I)\subseteq\text{Ann}(J)$ and $\text{Ann}(J)\subseteq\text{Ann}(I)$. Since $\text{Ann}(I)=\text{Ann}(J)$ we have the path

\[
I\leftrightarrow\text{Ann}(I)\leftrightarrow J.
\]
\end{proof}

We now recover the famous result of D. F. Anderson and P. Livingston.

\begin{cor}[D. F. Anderson and P. Livingston,\cite{andersonlivingston}]
If $Z(R)$ is the zero divisor graph of a commutative ring with identity, then $Z(R)$ is connected with diameter no more than 3.
\end{cor}

\begin{proof}
In the proof of Theorem \ref{gen}, we now restrict to the case where $I$ and $J$ are principal (and note that principal subideals can be chosen from within each respective annihilator). This shows that the graph of {\it principal ideals} that are generated by zero-divisors satisfy the conclusion of the corollary. To see that the slightly different statement formulated by Anderson and Livingston holds, just note that if $a,b\in R$ are distinct nonzero zero-divisors such that $(a)\neq (b)$, then there is a path among nonzero principal ideals of the form

\[
(a)\leftrightarrow (x)\leftrightarrow (y)\leftrightarrow (b)
\]

\noindent or

\[
(a)\leftrightarrow (x)\leftrightarrow (b).
\]

This yields the path $a\leftrightarrow x\leftrightarrow y\leftrightarrow b$ or $a\leftrightarrow x\leftrightarrow b$. In the case $(a)=(b)$, we note that there is an element $0\neq c$ such that $ac=0=bc$ giving the path $a\leftrightarrow c\leftrightarrow b$ unless $a=c$ (without loss of generality). But in this case, $a^2=0=ab$ and we are done.
\end{proof}

We now return to $GZ(R)$. Note that if $R$ contains a regular element $x$ then the ideal $(x)$ is an isolated vertex (and so, in general, one does not expect this graph to be connected). We highlight a case where connection is forced.

\begin{thm}
Let $R$ be a commutative ring with identity that is not a field. If $GA_1(R)$ or $GA_2(R)$ is connected, then so is $GZ(R)$. Additionally
$\text{diam}(GZ(R))\leq 3$.
\end{thm}

\begin{proof}
By Theorem \ref{conn} there is a collection of not necessarily distinct maximal ideals
$\{\mathfrak{M}_1,\mathfrak{M}_2,\cdots ,\mathfrak{M}_n\}$ such that
$\mathfrak{M}_1\mathfrak{M}_2\cdots \mathfrak{M}_n=0,\ n\geq 2$. We can also assume that no proper
subproduct of these listed ideals is zero.

If $I,J\subset R$ then there is a maximal ideal $\mathfrak{M}$ containing $I$. Since
$\mathfrak{M}\supseteq I\supset 0=\mathfrak{M}_1\mathfrak{M}_2\cdots \mathfrak{M}_n$,
$\mathfrak{M}$ must be $\mathfrak{M}_i$ for some $1\leq i\leq n$. We will assume without loss
of generality that $I\subseteq\mathfrak{M}_1$ and $J\subseteq\mathfrak{M}_k,\ 1\leq k\leq n$. To see that
$GZ(R)$ is connected with diameter no more than three consider the path for the case where $k\neq 1$:

\[
I\leftrightarrow\mathfrak{M}_2\mathfrak{M}_3\cdots\mathfrak{M}_n\leftrightarrow\mathfrak{M}_1\mathfrak{M}_2\cdots\mathfrak{M}_{k-1}\mathfrak{M}_{k+1}\cdots\mathfrak{M}_{n}\leftrightarrow J
\]

In the case that $k=1$, we modify our path as follows:

\[
I\leftrightarrow\mathfrak{M}_2\mathfrak{M}_3\cdots\mathfrak{M}_n\leftrightarrow J.
\]
\end{proof}

We conclude this section with an example to show that $GZ(R)$ may be connected without $GA_1(R)$, $GA_2(R)$ being connected.

\begin{ex}
Let $\mathbb{F}$ be a field and consider first the domain $R:=\mathbb{F}[x_1, x_2,\cdots, x_n,\cdots, y]$, let $I\subset R$ be the ideal $I:=(\{x_i^2, x_iy,y^2\}_{i\geq 1})$, and let $T:=R/I$ (we abuse the notation by now thinking of the elements $y, x_i$ as elements of $T$). $T$ is quasilocal, with maximal ideal $\mathfrak{M}$. It is easy to see that $GZ(T)$ is connected. Indeed if $I, J\subset T$ then we have the path $I\leftrightarrow (y)\leftrightarrow J$. But there is no collection of maximal ideals with product $0$ since for all $k\geq 1,\ 0\neq x_1x_2\cdots x_k\in\mathfrak{M}^k$. 
\end{ex}
\section{Finite Containment: $GF(R)\text{ and }GP(R)$}

In this section we investigate graphs of rings where the edges are defined to highlight finite (or perhaps principal) generation of an ideal over a subideal that it contains. In these graphs, the vertices will be the set of proper ideals and edges between ideals will be defined by finite generation. Specifically, we will declare that the ideals $I$ and $J$ have an edge between them if $I\subset J$ and $J$ is finitely (resp. principally) generated over $I$; that is, $J=(I,x_1,x_2,\cdots, x_n)$ (resp. $J=(I,x)$).

The algebraic motivation for this definition is an attempt to measure properties of $\text{MaxSpec}(R)$ in a condition that (under certain graphical constraints) mimics the Noetherian condition.

This first lemma is very easy and is presented to connect our definition of finite generation of one ideal over another to what was presented in Definition \ref{fcg}. We record it for completeness and omit the proof.

\begin{lem}
If $I\subseteq J\subset R$ be ideals of $R$, then $J/I$ is a finitely generated ideal of $R/I$ if and only if $J$ is finitely generated over $I$.
\end{lem}

Here are some preliminary containments of note.

\begin{prop}
$GA_0(R)$ is a subgraph of $GP(R)$ and $GP(R)$ is a subgraph of $GF(R)$.
\end{prop}

\begin{proof}
Only the first statement needs proof. Suppose that $I\subsetneq J\subsetneq R$ are adjacent ideals (so there is an edge between them in $GA_0(R)$). If $x\in J\setminus I$ then $J=(I,x)$ by adjacency and hence $I$ and $J$ have an edge between them in $GP(R)$.
\end{proof}

We now present the quasilocal case, which we will see is an exceptional case for these graphs in the sense that they are not always connected. Perhaps surprisingly, outside the quasilocal case, $GF(R)$ and $GP(R)$ are always connected.

\begin{thm}\label{QL}
Let $(R,\mathfrak{M})$ be quasilocal. $GF(R)$ (resp. $GP(R)$) is connected if and only if $\mathfrak{M}$ is finitely generated. Additionally, we have the following.
\begin{enumerate}
\item If $GF(R)$ is connected then $\text{diam}(GF(R))\leq 2$ and $GF(R)$ is complete if and only if $R$ is a Noetherian chained ring of dimension no more than $1$.
\item If $GP(R)$ is connected then $\text{diam}(GP(R))\leq 2n$ where $n$ is the minimal number of generators required for $\mathfrak{M}$ and $GP(R)$ is complete if and only if $R$ is a PIR.
 \end{enumerate}
\end{thm}

\begin{proof}
We first remark that if $R$ is a field then the result hold trivially, so we will assume that $R$ is not a field.
For the initial statement, we first suppose that $\mathfrak{M}$ is finitely generated. To show that $GF(R)$ is connected, we first note that since $\mathfrak{M}$ is finitely generated, it is certainly finitely generated over $I$. Hence if $I,J\subset R$ are any two ideals of $R$, then $I\leftrightarrow\mathfrak{M}\leftrightarrow J$ is a path of length no more than 2 from $I$ to $J$. For $GP(R)$ the proof is similar (with a possibly longer path).

Now we suppose that $GF(R)$ is connected. By assumption, there is a finite path from $(0)$ to $\mathfrak{M}$. Since a finite extension of a finite extension is finite, we can assume that this path takes on the form

\[
(0)\subset I_1\supset J_1\subset I_2\supset J_2\subset\cdots\subset I_{n-1}\supset J_{n-1}\subset I_n=\mathfrak{M}.
\]

\noindent Note that $I_1=(a_{1,1},a_{1,2},\cdots,a_{1,t_1})$ is finitely generated and for all $2\leq m\leq n$ we can write $I_m=(J_{m-1},a_{m,1},a_{m,2},\cdots,a_{m,t_m})$, and since $J_{m-1}\subset I_{m-1}$, we have that $I_m\subseteq (I_{m-1}, a_{m,1},a_{m,2},\cdots,a_{m,t_m})$

 In particular,
 
 \[
 I_n=\mathfrak{M}\subseteq (I_1, \{a_{i,j}\}_{j=1}^{t_i}{_{i=2}^n})
 \]
 
\noindent and since $\mathfrak{M}$ is maximal, equality holds. Hence $\mathfrak{M}$ is finitely generated. Since each finite extension is a finite sequence of principal extensions, this establishes the statement for $GP(R)$ as well.

For (1), we have already established that $\text{diam}(GF(R))\leq 2$. If $GF(R)$ is complete, then any two ideals must be comparable and hence $R$ is chained. Now note that if $I$ is an arbitrary ideal of $R$ then there is an edge between $I$ and $(0)$ and so $I$ must be finitely generated and so $R$ is Noetherian. Finally, we note that since $R$ is Noetherian and chained, its dimension must be no more than $1$ (if $R$ has a height $1$ nonmaximal prime ideal then there must be infinitely many and hence $R$ cannot be chained).

For the converse, note that if $R$ is Noetherian and chained then any two ideals are comparable and since they are finitely generated, there must be an edge between them.

For (2) we note that if $\mathfrak{M}$ is generated by $\{x_1,x_2,\cdots, x_n\}$ then as $R$ is quasilocal, there is a path from any ideal $I$ to $\mathfrak{M}$ of length no more than $n$. So given ideals $I,J\subset R$, there is a path from $I$ to $J$ bounded by twice the number of generators of $\mathfrak{M}$ and hence $\text{diam}(GP(R))\leq 2n$.

Now suppose that $GP(R)$ is complete and let $I\subset R$ be an ideal. Since $I$ and $(0)$ have an edge between them, $I$ must be principal. Conversely if $R$ is a PIR and local then $R$ is chained. To see this, note that if $I=(x)$ and $J=(y)$ then $(d)=(x,y)$. If $dx'=x$ and $dy'=y$, it is easy to see that $(x',y')=R$ and so (without loss of generality) $x'$ is a unit and $x\vert y$. Hence $J\subseteq I$ and so there is an edge between $I$ and $J$.
\end{proof}

We now suppose that $R$ is not quasilocal. In this case, $GF(R)$ and $GP(R)$ are always connected, but as we will see, the diameter of $GF(R)$ reveals some subtleties concerning the structure of $R$. With regard to diameter, we will focus on $GF(R)$ as the large number of steps sometimes required to make a path in $GP(R)$ clouds the issue a bit. On the other hand, we will make note of $GP(R)$ when prudent.

 We begin with a definition and a useful lemma that will simplify matters.

\begin{defn}
Let $I\subset R$ be an ideal. We define $\text{MaxSpec}_I(R)$ to be the collection of maximal ideals of $R$ containing $I$ and following \cite{G} we define the Jacobson radical of $I$, $\mathfrak{J}(I)$, to be the intersection of all maximal ideals of $R$ containing $I$. If $R$ is a ring we use the notation $\mathfrak{J}(R)$ to be the Jacobson radical of $R$.
\end{defn}

We remark that it is an easy exercise to verify that $\mathfrak{J}(R/I)=\mathfrak{J}(I)/I$ and we will use this fact on a number of occasions.

In the following key lemma, we describe paths between ideals in $GF(R)$ and $GP(R)$. The bounds apply to both as principally generated ideals will be used in the proof.

\begin{lem}\label{num}
Let $R$ be a commutative ring with identity and consider the graphs $GF(R)$ and $GP(R)$.
\begin{enumerate}
\item If the ideals $I$ and $J$ are comaximal, then there is a path of length $2$ from $I$ to $J$.
\item If $I$ and $J$ are not comaximal, but there is a maximal ideal $M$ that contains $I$ but not $J$ then there is a path of length no more than $3$ from $I$ to $J$.
\item  If $\text{MaxSpec}_I(R)=\text{MaxSpec}_J(R)$ and $\vert\text{MaxSpec}(R)\vert>1$, then there is a path of length no more than $4$ from $I$ to $J$.
\end{enumerate}
\end{lem}

\begin{proof}
For (1), we suppose that $I$ and $J$ are comaximal and find $\alpha\in I$ and $\beta\in J$ such that $\alpha+\beta=1$. So if $x\in I$ (resp. $x\in J$) then the equation
\[
x\alpha+x\beta=x
\]
\noindent demonstrates that the ideal $I$ (resp. $J$) is singly generated by $\alpha$ (resp. $\beta$) over the ideal $IJ$. Hence in $GF(R)$ (as well as $GP(R)$), we have the path $I-IJ-J$ of length 2.  

For (2), since $M$ does not contain $J$ and is maximal, we can find $m\in M$, $j\in J$ such that $m+j=1$. Since $I\subseteq M$ we have that $(I,m)$ is a proper ideal and is clearly principally generated (or less) over $I$. By (1) there is a path between $(I,m)$ and $J$ of length $2$ and hence there is a path of length no more than $3$ from $I$ to $J$.


For the last statement, we will assume that $\text{MaxSpec}_I(R)=\text{MaxSpec}_J(R)$ and select $M\in\text{MaxSpec}_I(R)$. By hypothesis, there is another maximal ideal $N$ and we find $m\in M$ and $n\in N$ such that $m+n=1$. In any case, we have the ideals $I$ and $(I, m)$ are equal or have an edge between them. If $J\subseteq N$ then $J$ and $(J,n)$ are equal or have an edge between them and since $(I,m)$ and $(J,n)$ are comaximal, we have our desired path of length no more than 4. On the other hand, if $J$ is not contained in $N$, then $J$ and $N$ are comaximal (as are $I$ and $N$). From part (1) there is a path of length 2 from $I$ to $N$ and a path of length 2 from $N$ to $J$ and this completes the proof.
\end{proof}

The following theorem follows directly from Lemma \ref{num}.

\begin{thm}
Let $R$ be commutative with 1 with $\vert\text{MaxSpec}(R)\vert>1$. Then $GF(R)$ and $GP(R)$ are connected and $\text{diam}(R)\leq 4$.
\end{thm}

We now focus on $GF(R)$ in the case that $\vert\text{MaxSpec}(R)\vert>1$ and examine necessary and sufficient conditions for the diameter to be of prescribed sizes.

\begin{thm}\label{66}
Let $R$ be a commutative ring with 1 with $\vert\text{MaxSpec}(R)\vert>1$.
\begin{enumerate}
\item $\text{diam}(GF(R))= 2$ if and only if every maximal ideal of $R$ is finitely generated.
\item $\text{diam}(GF(R))\leq 3$ if and only if given $I,J\subset R$ with $\mathfrak{J}(I)=\mathfrak{J}(J)$ then there is a proper ideal $K$ that is finitely generated over both $I$ and $J$.
\end{enumerate}
\end{thm}

\begin{proof}
If $\text{diam}(GF(R))= 2$, then given any maximal ideal $\mathfrak{M}$, there is a path of length $1$ or $2$ between $\mathfrak{M}$ and $(0)$. Whether the path is of the form $(0)\relbar\mathfrak{M}$ or $(0)\relbar I\relbar\mathfrak{M}$, $\mathfrak{M}$ is finitely generated.

On the other hand, if $I,J\subset R$ are ideals, then there is a path of length $2$ between them if $I$ and $J$ are comaximal by Lemma \ref{num}. If $I$ and $J$ are both contained in the maximal ideal $\mathfrak{M}$ then $I\leftrightarrow \mathfrak{M}\leftrightarrow J$ is a path of length no more than $2$ between them. Since $\vert\text{MaxSpec}(R)\vert>1$, we have that $\text{diam}(GF(R))$ is precisely $2$. This establishes the first statement.

For (2), we first suppose that $\text{diam}(GF(R))\leq 3$. Let $I,J\subset R$ be ideals with $\mathfrak{J}(I)=\mathfrak{J}(J)$ and consider cases. 
If there is an edge between $I$ and $J$, then the larger ideal is finitely generated over both. 

If there is a path of length $2$ between $I$ and $J$, say $I-A-J$ and $I,J\subset A$ then $A$ is finitely generated over both $I$ and $J$. If on the other hand, $A\subset I,J$, then $I=(A,x_1,x_2,\cdots, x_n)$ and $J=(A,y_1,y_2,\cdots, y_m)$. Now note that $I+J=(I, y_1,y_2,\cdots, y_m)=(J,x_1,x_2,\cdots, x_n)$ is finitely generated over both $I$ and $J$ and note that since $\mathfrak{J}(I)=\mathfrak{J}(J)$, $I+J$ is proper.

Finally, we suppose that there is a path of length $3$ between $I$ and $J$ of the form $I-A-B-J$. Without loss of generality we will assume that we have the containments $I\supset A\subset B\supset J$. In a similar fashion to the previous argument, we write $I=(A,x_1,x_2,\cdots,x_n)$ and $B=(A, z_1,z_2, \cdots, z_t)=(J,y_1,y_2,\cdots, y_m)$. Note that $I+B$ is a proper ideal as $\mathfrak{J}(I)=\mathfrak{J}(J)\subseteq\mathfrak{J}(B)$. So the ideal $I+B=(I,z_1,z_2,\cdots, z_t)=(J, y_1,y_2,\cdots, y_m,x_1,x_2,\cdots, x_n)$ is finitely generated over both $I$ and $J$.

Conversely, note that Lemma \ref{num} shows that the only unresolved case is the case in which $I,J\subset R$ with $\mathfrak{J}(I)=\mathfrak{J}(J)$. But the existence of $K$ shows that there is a path between $I$ and $J$ of length no more than $2$. Hence $\text{diam}(GF(R))\leq 3$.
\end{proof}

\begin{cor}\label{FG}
Let $R$ be commutative with identity. Then $\text{diam}(GF(R))\leq 3$ if and only if for all $I\subset R$, $\mathfrak{J}(R/I)$ is contained in a finitely generated proper ideal.
\end{cor}

\begin{proof}
Note first that Theorem \ref{QL} shows that this result holds in the quasilocal case, and so we will assume that $\vert\text{MaxSpec}(R)\vert>1$.

We suppose first that $\text{diam}(GF(R))\leq 3$ and select an arbitrary $I\subset R$. Theorem \ref{66} assures us that there is an ideal $B\subset R$ with $B$ finitely generated over both $I$ and $\mathfrak{J}(I)$ and so $B/I$ is a finitely generated ideal of $R/I$ that contains $\mathfrak{J}(I)/I=\mathfrak{J}(R/I)$.

Conversely, suppose that for all $I\subset R$ we have that $\mathfrak{J}(R/I)$ is contained in a finitely generated ideal and select ideals $I, J\subset R$ with $\mathfrak{J}(I)=\mathfrak{J}(J)$. By assumption $\mathfrak{J}(R/IJ)=\mathfrak{J}(IJ)/IJ$ is contained in the finitely generated ideal $K/IJ$. Note that as $\text{MaxSpec}_I(R)=\text{MaxSpec}_J(R)$, $\mathfrak{J}(IJ)=\mathfrak{J}(J)=\mathfrak{J}(I)$ and so we have that $K/IJ$ contains $\mathfrak{J}(I)/IJ$ which in turn contains $I/IJ$. Hence $K/I\cong(K/IJ)/(I/IJ)$ is finitely generated. A similar proof establishes that $K/J$ is finitely generated, and so $K$ is finitely generated over both $I$ and $J$. By Theorem \ref{66} we are done.
\end{proof}

\begin{ex}
Let $V$ be a $1-$dimensional nondiscrete valuation domain with maximal ideal $\mathfrak{M}$ and consider the ring $V[x]$. Note that $V\cong V[x]/(x)$ has Jacobson radical $\mathfrak{M}$ and so by Corollary \ref{FG}, $GF(V[x])$ has diameter $4$.
\end{ex}

\begin{ex}\label{DP}
Consider the ring $R:=\prod_{i\in\Gamma}\mathbb{F}_2$ where $\Gamma$ is a nonempty index set with at least two elements (the graph $GF(R)$ is a single vertex in the degenerate case that $\Gamma$ consists of a single element). Note first that any element of $R$ is idempotent and so this is true of any homomorphic image of $R$. So if $T$ is a homomorphic image of $R$ and $e\in T$ then $e(1-e)=0$. So if $e\in\mathfrak{J}(T)$ then $1-e$ is a unit in $T$ which implies that $e=0$. Since $\mathfrak{J}(T)=0$ for all homomorphic images of $T$ of $R$, we have that $\text{diam}(GF(R))=3$ if $\Gamma$ is infinite and $\text{diam}(GF(R))=2$ if $\Gamma$ is finite with at least two elements.
\end{ex}

We now wish to show that the family of almost Dedekind domains that are not Dedekind produce examples of domains where $\text{diam}(GF(R))=3$ in abundance. We first require a couple of preliminary results.

\begin{lem}\label{jr}
If $R$ is a commutative ring with 1, $S\subseteq R$ a multiplicative set, and $I\subset R$ an ideal with the property that for all $s\in S$, there is an $x_s\in I$ and $s^\prime\in S$ such that $s^\prime s+x_s=1$, then $R/I\cong R_S/I_S$.
\end{lem}

\begin{proof}
We define $\phi:R\longrightarrow R_S/I_S$ given by $\phi(z)=\frac{z}{1}+I_S$. We suppose that $r\in R$ and $s\in S$. By hypothesis, we can find $x_s\in I, s^\prime\in S$ with $s^\prime s+x_s=1$. So in $R_S$ we have $s'+\frac{x_s}{s}=\frac{1}{s}$ and hence $rs'+\frac{rx_s}{s}=\frac{r}{s}$.

We now observe that $\phi(rs')=\frac{rs'}{1}+I_S=\frac{r}{s}-\frac{rx_s}{s}+I_S=\frac{r}{s}+I_S$ and so $\phi$ is onto.

It is easy to see that $I\subseteq\text{ker}(\phi)$. For the other containment, note that if $z\in\text{ker}(\phi)$, then $\frac{z}{1}=\frac{\alpha}{s}$ for some $\alpha\in I$ and $s\in S$. Hence there exists $t\in S$ such that $t(zs-\alpha)=0$. Using the assumed property, there is a $t^\prime\in S$ and $\beta\in I$ such that $tt^\prime+\beta=1$. From this we obtain that $(1-\beta)(zs-\alpha)=0$ and so $zs\in I$. Once again, we use the fact that there is $s^\prime\in S$ and $\gamma\in I$ such that $ss^\prime+\gamma=1$ and we now obtain that $z(1-\gamma)\in I$ and hence $z\in I$. By the First Isomorphism Theorem, we have that $R/I\cong R_S/I_S$.
\end{proof}

\begin{prop}\label{LC}
Let $I\subset R$ be an ideal and $\{M_i\}_{i\in\Gamma}$ be the collection of maximal ideals of $R$ that contain $I$. If $S=(\bigcup_{i\in\Gamma} M_i)^c$, then $R/I\cong R_S/I_S$.
\end{prop}

\begin{proof}Suppose that $s\in S$ and note that by the definition of $S$, $(s,I)=R$. Hence we can find $x_s\in I, s^\prime\in R$  with $s^\prime s+x_s=1$. Observe that $s'$ cannot be in any maximal ideal containing $I$ and so must be in $S$; we now appeal to Lemma \ref{jr}.
\end{proof}

\begin{prop}\label{AD}
If $R$ be an almost Dedekind domain that is not Dedekind with finitely many maximal ideals that are not finitely generated, then $\text{diam}(GF(R))=3$.
\end{prop}

We remark that almost Dedekind domains meeting the above conditions are commonplace. In particular, consider the sequence domains defined in \cite{L2006}. 

\begin{proof}
For the first case, we suppose that $I$ is contained in only finitely many maximal ideals and $M_1,M_2,\cdots,M_n$ be the maximal ideals containing $I$. If $S=(\bigcup_{i=1}^n M_i)^c$, then Proposition \ref{LC} shows that $R/I=R_S/I_S$, but as $R_S$ is an almost Dedekind domain with only finitely many maximal ideals, it is Dedekind and hence $\mathfrak{J}(R/I)\cong\mathfrak{J}(R_S/I_S)$ is finitely generated. Hence by Corollary \ref{FG}, $\text{diam}(GF(R))\leq 3$.

Now suppose that $I$ is contained in infinitely many maximal ideals. Note that since there are only finitely many maximal ideals that are not finitely generated, there must be infinitely many finitely generated primes (if not then $R$ is semiquasilocal and hence Dedekind). So if $M$ is a finitely generated maximal ideal containing $I$, then $M/I$ is a finitely generated ideal of $R/I$ containing $\mathfrak{J}(R/I)=\mathfrak{J}(I)/I$ and so again $\text{diam}(GF(R))\leq 3$ by Corollary \ref{FG}.

Equality follows from the existence of maximal ideals that are not finitely generated and Theorem \ref{66}.
\end{proof}

We also distinguish behavior in the Noetherian case; we ignore the case of a field as the graph in this case is a single vertex.

\begin{thm}
Let $R$ be a commutative ring with identity that is not a field. The following conditions are equivalent.
\begin{enumerate}
\item $R$ is Noetherian.
\item The radius of $GF(R)$ is equal to $1$, and $(0)$ is a center of $GF(R)$.
\item $\text{diam}(GF(R))\leq 2$ and if $I,J\subset R$ then either $I$ and $J$ have an edge between them or there is a minimal path between them passing through $(0)$.
\end{enumerate}
\end{thm}

\begin{proof}
For $(1)\Longrightarrow (2)$ we assume that $R$ is Noetherian. In this case there is an edge between $(0)$ and an arbitrary nonzero ideal. So $(0)$ is a center and the radius of $GF(R)$ is precisely $1$.

For $(2)\Longrightarrow (3)$, since $(0)$ is a center and has an edge with an arbitrary ideal, every ideal is finitely generated. In particular, each maximal ideal is finitely generated and so $\text{diam}(GF(R))\leq 2$ by Theorem \ref{66}. Also note that from the above, if $I$ and $J$ do not have an edge between them, then we have the path $I\leftrightarrow(0)\leftrightarrow J$.

Finally, for $(3)\Longrightarrow (1)$, we suppose that $I\subset R$ and consider the pair of ideals $I$ and $(0)$. In either case there is an edge between them, and so $I$ is finitely generated.
\end{proof}

We now give, in sequence, theorem for the behavior of $GF(-)$ for polynomial extensions, power series extensions, and homomorphic images. For the statement of these theorems, the case $n=0$ corresponds to the case in which $R$ is a field (and the graph being a single vertex).

\begin{thm}\label{poly}
Suppose $R$ is a commutative ring with 1. If $\text{diam}(GF(R))=\infty$ then $\text{diam}(GF(R[x_1,x_2,\cdots,x_m]))=4$. If $\text{diam}(GF(R))=n<\infty$ then $\text{diam}(GF(R))\leq \text{diam}(GF(R[x_1,x_2,\cdots, x_m]))$ and the following hold.
\begin{enumerate}
\item If $n\leq 1$, then $\text{diam}(GF(R[x_1,x_2,\cdots,x_m]))=2$.
\item If $R$ is Noetherian, then $\text{diam}(GF(R[x_1,x_2,\cdots,x_m]))=2$.
\item If $n=4$, then $\text{diam}(GF(R[x_1,x_2,\cdots,x_m]))=4$.
\end{enumerate}
\end{thm}

\begin{proof}
For the first statement, if $\text{diam}(GF(R))=\infty$ then $R$ is quasilocal with maximal ideal $\mathfrak{M}$ that is not finitely generated. Now note that $R[x_1,x_2,\cdots, x_n]/(x_1,x_2,\cdots,x_n)\cong R$. Since $\mathfrak{J}(R)=\mathfrak{M}$ is not contained in a finitely generated ideal, we have produced an ideal, namely $(x_1,x_2,\cdots,x_n)$ such that $\mathfrak{J}(R[x_1,x_2,\cdots, x_n]/(x_1,x_2,\cdots,x_n))$ is not contained in a finitely generated ideal and so by Corollary \ref{FG}, $\text{diam}(GF(R[x_1,x_2,\cdots,x_n]))=4$.

For the case $\text{diam}(GF(R))=n<\infty$, we first note that in the case that $\text{diam}(GF(R))\leq 2$ the fact that $R[x_1,x_2,\cdots,x_n]$ is not chained shows by application of Theorem \ref{QL}, that $\text{diam}(GF(R[x_1,x_2,\cdots,x_n]))\geq 2$. If $\text{diam}(GF(R))=3$, then there is a maximal ideal $\mathfrak{M}$ that is not finitely generated. Since the ideal $(\mathfrak{M},x_1,x_2,\cdots,x_n)\subset R[x_1,x_2,\cdots,x_n]$ is not finitely generated, $\text{diam}(GF(R[x_1,x_2,\cdots,x_n]))\geq 3$. Finally, if $\text{diam}(GF(R))=4$, Corollary \ref{FG} gives that there is an ideal $I\subset R$ such that $\mathfrak{J}(R/I)$ is not contained in a finitely generated ideal. As $R[x_1,x_2,\cdots, x_n]/(I,x_1,x_2,\cdots,x_n)\cong R/I$, Corollary \ref{FG} again applies and $\text{diam}(GF(R[x_1,x_2,\cdots,x_n]))=4$.

For $(1)$, if $n=1$ then by Theorem \ref{QL}, $R$ is Noetherian and chained (this statement also holds in the case where $n=0$ and $R$ is a field). Hence $R[x_1,x_2,\cdots,x_n]$ is Noetherian but has infinitely many maximal ideals (so is not chained), therefore by Theorem \ref{66} $\text{diam}(GF(R[x_1,x_2,\cdots,x_n]))=2$. 

For $(2)$, if $R$ is Noetherian, then $R[x_1,x_2,\cdots,x_n]$ is Noetherian and the result follows as a porism of the proof of $(1)$. 

$(3)$ is now immediate.
\end{proof}

\begin{ex}\label{16}
Let $\displaystyle R=\mathbb{F}[x,\frac{y}{x},\frac{y}{x^2},\cdots ]$ where $\mathbb{F}$ is a field and let $V=R_{\mathfrak{M}}$ where $\displaystyle \mathfrak{M}=(x,\frac{y}{x},\frac{y}{x^2},\cdots)$. $V$ is a 2-dimensional
discrete valuation domain with principal maximal ideal (and hence has the property that all of
its maximal ideals are finitely generated); its prime spectrum is $\mathfrak{M}V=(x)\supset\mathfrak{P}=(y,\frac{y}{x},\frac{y}{x^2},\cdots)\supset (0)$. But if we consider the polynomial ring $V[t]$, it
is easy to see that the ideal $\mathfrak{N}:=(\mathfrak{P}[t], xt+1)$ is maximal, but
not finitely generated. Note that the ideal $(\mathfrak{P}^2[t], xt+1)$ has radical $\mathfrak{N}$ and hence $V[t]/(\mathfrak{P}^2[t], xt+1)$ has Jacobson radical $\mathfrak{N}/(\mathfrak{P}^2[t], xt+1)$ but $\mathfrak{N}$ is not finitely generated over $(\mathfrak{P}^2[t], xt+1)$. Thus $\text{diam}(GF(V))=2$ but $\text{diam}(GF(V[t]))=4$. These details will also follow from the following.
\end{ex}

From a more global perspective (in contrast with the class of almost Dedekind domains discussed in Proposition \ref{AD}), the condition of Corollary \ref{FG} appears to make the situation where the diameter of $GF(R)$ is precisely $3$ a reasonably rare occurrence. For example, we consider the class of SFT (for strong finite type) rings introduced by J. Arnold in \cite{Ar1973} as a generalization of Noetherian rings useful in the study of the dimension of power series rings. An ideal $I\subseteq R$ is said to be SFT if there is a finitely generated ideal $B\subseteq I$ and a fixed natural number $m$ such that $x^m\in B$ for all $x\in I$; we say the ring, $R$, is SFT if all of its ideals are SFT.

\begin{prop}\label{SFT}
Let $R$ be SFT, or more generally, any ring with the property that each maximal ideal is the radical of a finitely generated ideal. Then $\text{diam}(GF(R))\neq 3$.
\end{prop}

\begin{proof}
We can assume that there is a maximal ideal $\mathfrak{M}$ that is not finitely generated. By assumption, $\mathfrak{M}$ is the radical of a finitely generated ideal $B\subset\mathfrak{M}$. Note that $\mathfrak{J}(R/B)=\mathfrak{J}(B)/B=\mathfrak{M}/B$ and is maximal. If $\mathfrak{M}/B$ is finitely generated, the fact that $B$ is finitely generated implies that $\mathfrak{M}$ is finitely generated. Hence $\mathfrak{J}(R/B)$ is not contained in a finitely generated ideal and so by Corollary \ref{FG}, $\text{diam}(GF(R))=4$.
\end{proof}

It is worth noting that $\text{diam}(GF(R))=1,2,4, \infty$ can occur for SFT rings. Also, Proposition \ref{SFT} is another route to Example \ref{16}. Indeed, the domain $V$ from this example is a $2-$dimensional SFT valuation domain and the results of \cite{kpmix} show that $\text{dim}(V[t][[y]])$ is of finite Krull dimension and hence $V[t]$ must be SFT. Since the maximal ideal $\mathfrak{N}$ is not finitely generated, $\text{diam}(GF(V[t]))=4$.

\begin{thm}
If $R$ is a commutative ring with 1 and $\text{diam}(GF(R))=n$ then the following hold.
\begin{enumerate}[start=0]
\item If $n=0$, then $\text{diam}(GF(R[[x]]))=1$.
\item If $n=1$, then $\text{diam}(GF(R[[x_1,x_2,\cdots,x_m]]))=2$.
\item If $n=2$, then $\text{diam}(GF(R[[x_1,x_2,\cdots,x_m]]))= 2$.
\item If $n=3,4, \infty$, then $\text{diam}(GF(R[[x_1,x_2,\cdots,x_m]]))=n$. 
\end{enumerate}
\end{thm}

\begin{proof}
For $(0)$ we note that if $R$ is a field then $R[[x]]$ is a Noetherian valuation domain. 

For $(1)$, we have by Theorem \ref{QL} that $R$ is Noetherian and quasilocal and hence so is $R[[x_1,x_2,\cdots,x_n]]$ (but not chained if $R$ is not a field) and so $\text{diam}(GF(R[[x_1,x_2,\cdots,x_n]]))=2$ by Theorem \ref{QL}.

$(2)$ is similar to the previous. Since each maximal ideal is finitely generated and the maximal ideals of $R[[x_1,x_2,\cdots,x_n]]$ are of the form $(\mathfrak{M},x_1,x_2,\cdots, x_n)$, we have that all maximal ideals of $R[[x_1,x_2,\cdots,x_n]]$ are finitely generated as well.

For $(3)$, we first consider $n=3$. Suppose that $R$ has the property that for all $I\subset R$, $\mathfrak{J}(R/I)$ is contained in a finitely generated ideal. Now consider an ideal $A\subset R[[x]]$ and let $A_0=\{f(0)\vert f(x)\in A\}$. An easy computation shows that $\mathfrak{J}(A)$ is precisely equal to $(\mathfrak{J}(A_0),x)$ where $\mathfrak{J}(A_0)$ is the Jacobson radical of $A_0$ in $R$. Note that there is an ideal $F\subset R$ such that $F$ is finitely generated over $\mathfrak{J}(A_0)$ and $A_0$ (that is, $F/A_0$ is a finitely generated ideal containing $\mathfrak{J}(A_0)/A_0=\mathfrak{J}(R/A_0)$). We conclude that $(F,x)$ is finitely generated over $\mathfrak{J}(A)=(\mathfrak{J}(A_0),x)$ and with this in hand, we will show that $(F,x)$ is finitely generated over $A$. To this end, we let $(y_1,y_2,\cdots, y_n)$ be the generators of $F$ over $A_0$ and for each $1\leq i\leq n$, find $a_i(x)\in A$ such that $a_i(0)=y_i$. Now suppose that $\alpha p_1(x)+xp_2(x)\in (F,x)$, with $\alpha\in F$ and $p_1(x), p_2(x)\in R[[x]]$. Since $x$ is a generator, we can assume, without loss of generality, that $p_1:=p_1(x)\in R$. Since $\alpha p_1\in F$, we can write $\alpha p_1=r_1y_1+r_2y_2+\cdots +r_ny_n+a_0$ with $r_i\in R, a_0\in A_0$. Note that $\alpha p_1+xp_2(x)=r_1a_1(x)+r_2a_2(x)+\cdots +r_na_n(x)+a(x)+xg(x)$ where $a(x)\in A$ with $a(0)=a_0$ and $g(x)\in R[[x]]$. So $(F,x)$ is generated over $A$ by the elements $a_i(x)$ and $x$. Hence by Corollary \ref{FG} $\text{diam}(R[[x]])=3$. The multivariable case follows immediately by induction.

If $n=4$ then there is an ideal $I\subset R$ such that $\mathfrak{J}(R/I)$ is not contained in a finitely generated ideal. Note that as $R[[x]]/(I,x)\cong R/I$ this property persists in the power series extension. Again, induction completes this.

Finally, if $n=\infty$, $R$ is quasilocal with maximal ideal $\mathfrak{M}$ that is not finitely generated by Theorem \ref{QL}. Since $R[[x_1,x_2,\cdots,x_n]]$ is quasilocal with maximal ideal $(\mathfrak{M}, x_1,x_2,\cdots, x_n)$ the diameter of $GF(R[[x_1,x_2,\cdots,x_n]])$ is also infinite.
\end{proof}

We round this out by recording behavior in homomorphic images.

\begin{thm}
Let $R$ be a commutative ring with 1 and $I\subset R$ an ideal. If $\text{diam}(GF(R))=n$ then the following hold.
\begin{enumerate}
\item If $n\leq 2$, then $\text{diam}(GF(R/I))\leq n$.
\item If $n=3$, then $\text{diam}(GF(R/I))\leq n$.
\item If $n=4$, then $\text{diam}(GF(R/I))\leq n$ or  $\text{diam}(GF(R/I))=\infty$.
\end{enumerate}
\end{thm}

\begin{proof}
If $n\leq 2$ it is necessarily true that the maximal ideals of $R$ are finitely generated and hence the same is true of $R/I$. 

For the second statement, we suppose that $\text{diam}(GF(R))=3$, and so in particular, for all $I\subset R$, $\mathfrak{J}(R/I)$ is contained in a finitely generated ideal. So if $J/I$ is an ideal of $R/I$ then $(R/I)/(J/I)\cong R/J$ and so $GF(R/I)$ has diameter of no more than $3$.

For $(3)$, if $\text{diam}(GF(R/I))$ is finite, then the first statement is clear. Note however, that if $R$ is a ring with a maximal ideal $\mathfrak{M}$ that is not finitely generated, and $I\subset\mathfrak{M}$ has the property that $\sqrt{I}=\mathfrak{M}$ and $\mathfrak{M}/I$ is not finitely generated, then $R/I$ is quasilocal and hence its graph is not connected by Theorem \ref{QL}. For a concrete example of this, let $V$ be a $1-$dimensional nondiscrete valuation domain. By Theorem \ref{QL} $\text{diam}(GF(V))=\infty$ and $\text{diam}(GF(V[x]))=4$ by Theorem \ref{poly}; since $V\cong V[x]/(x)$ we have our example.
\end{proof}

\begin{rem}
As a final remark, we point out that it would be interesting to know if the property that $\text{diam}(GF(R))=3$ is stable under polynomial extensions. More generally, a more complete understanding of stability properties of the diameter of the graphs $GF(-)$ under polynomial extensions for $n=2,3$ would be desirable.
\end{rem}

\bibliography{biblio2}{}
\bibliographystyle{plain}

\end{document}